\documentclass[runningheads]{llncs}

\usepackage{graphicx}

\usepackage{amsmath,amsfonts,amssymb}
\usepackage{mathtools}

\usepackage{algorithm}
\usepackage{algorithmicx}
\usepackage[noend]{algpseudocode}

\algnewcommand{\IfThenElse}[3]{
  \State \algorithmicif\ #1\ \algorithmicthen\ #2\ \algorithmicelse\ #3}

\DeclareMathOperator*{\argmax}{arg\,max}

\usepackage{xspace}
\usepackage{booktabs}
\usepackage{xcolor}
\newtheorem{myprop}{Proposition}

\newtheorem{mydef}[myprop]{Definition}

\newtheorem{myex}[myprop]{Example}

\DeclareMathOperator*{\argmin}{arg\,min}

\newcommand{\IE}{\mbox{i.e.}\xspace}
\def\D_All{\mathcal{D}}

\usepackage{tabularx}
\usepackage{ragged2e}

\usepackage{url}
\newcolumntype{R}[1]{>{\raggedleft\arraybackslash}p{#1}}

\usepackage{numprint}
\npthousandsep{\,}

\usepackage{caption}
\usepackage{subcaption}

\newcommand{\USet}{\mathcal{L}_\forall}
\begin{document}

\title{Adaptive Relaxations  for Multistage Robust Optimization  \thanks{This research was partially funded by the Deutsche Forschungsgemeinschaft (DFG, German Research Foundation) - 399489083.}
}

\titlerunning{Adaptive Relaxations for Multistage Robust Optimization}
%
\author{Michael Hartisch}
\authorrunning{M. Hartisch }
%
\institute{University of Siegen, 57072 Siegen, Germany \\
\email{michael.hartisch@uni-siegen.de }}
\maketitle              
\begin{abstract}
Multistage robust optimization problems can be interpreted as two-person zero-sum games between two players. We exploit this game-like nature and utilize a game tree search in order to solve quantified integer programs (QIPs). In this algorithmic environment relaxations are repeatedly called to asses the quality of a branching variable and for the generation of bounds. A useful relaxation, however, must be well balanced with regard to its quality and its computing time. We present two relaxations that incorporate scenarios from the uncertainty set, whereby the considered set of scenarios is continuously adapted according to the latest information gathered during the search process.  Using selection, assignment, and runway scheduling problems as a testbed, we show the impact of our findings.

\end{abstract}
\section{Introduction}
Most aspects of decision making are highly affected by uncertainty. In order to take such uncertainty into account different methodologies have been developed, such as stochastic programming \cite{KM05} or robust optimization \cite{Ben-Tal}. In this setting, multistage models can be used to obtain an even more realistic description of the underlying problem. While there are several real multistage stochastic approaches (e.g. \cite{mercier2007performance,hemmi2018recursive}), extensions to robust optimization with more than two stages only recently gained more attention (e.g. \cite{delage2015robust,bertsimas2016multistage}). Due to their PSPACE-complete nature \cite{Papadimitriou}, tackling multistage robust problems is a very complicated task and for the human mind even comprehending a solution is rather challenging. Solution approaches 
 include approximation techniques \cite{bertsimas2015design}, dynamic programming \cite{shapiro2011dynamic}, and solving the deterministic equivalent problem (DEP), also referred to as robust counterpart \cite{Ben-Tal}, often using decomposition techniques (e.g. \cite{takriti2004robust}).  We, on the other hand, exploit the similarity of multistage robust problems with two-person zero-sum games and apply a game tree search to solve quantified integer programs (QIPs) \cite{Subramani_Integer,ederer2011quantified}. QIPs are integer linear programs with ordered variables that are either existentially or universally quantified, and provide a convenient framework for multistage robust optimization, allowing polyhedral or even decision-dependent uncertainty sets \cite{hartisch2019mastering}. The very intuitive approach of applying game tree search to solve the very compact QIP formulation paves the way for large multistage problems: a recent computational study showed that solving robust discrete problems with multiple stages is well within the reach of current computational prowess \cite{goerigk2021multistage}.

As in any tree search algorithm, a rapid but high-quality assessment of the potential of different subtrees is crucial for the search process. This can be done by relaxing some problem conditions in order to obtain a bound on the optimal value of a (sub)problem. 
In mixed integer linear programming (MIP), variants of the linear programming (LP)-relaxation of a problem are employed \cite{Balas2001}. Equivalently for QIPs, the quantified linear programming (QLP)-relaxation can be used. But its DEP's size remains exponentially large, even when tackled with decomposition techniques \cite{lorenz2015solving}. By further relaxing the variables' quantification the LP-relaxation of a QIP arises, which, however, completely neglects the problem's multistage and uncertain nature. In order to restore the robust nature of the problem, we exploit that a solution must cope with \textit{any} uncertain scenario: fixing (originally) universally quantified variables in this LP-relaxation yields a very powerful tool in our tree search algorithm. Furthermore, we show that if only a small subset of the uncertainty set is considered in the QLP-relaxation, the correseponding DEP remains small enough to yield an effective relaxation. 
This local approximation, which has similarites to sampling techniques \cite{GuptaA}, is utilized to eventually obtain the optimal solution for a multistage  robust optimization problem. 

For both enhanced relaxations the selection of incorporated scenarios crucially affects their effectiveness, i.e. having reasonable knowledge of which universal variable assignments are particularly vicious can massively boost the search process. We partially rely on existing heuristics, developed to analyze and find such promising assignments in a game tree search environment \cite{Killer,schaeffer1989history} as well as for solving SAT problems \cite{moskewicz2001chaff}. As these heuristic evaluations change over time, the relaxations adapt based on newly gathered information.

In Section \ref{Sec::QIP} we introduce the basics of quantified programming and outline the used game tree search. In Section \ref{Sec::Relax} we present the utilized relaxations and  we illustrate the strength of our approach in a computational study in Section \ref{Sec::Exp}, before we conclude in Section \ref{Sec::Concl}.

\section{Quantified Programming \label{Sec::QIP}}
In the following, we formally introduce quantified integer programming. \cite{DissMichael} can be consulted for a more detailed discussion.
\subsection{Basics of Quantified Integer Programming}

A QIP can be interpreted as a two-person zero-sum game between an \emph{existential player} setting the existentially quantified variables and a \emph{universal player} setting the universally quantified variables. The variables are set in consecutive order according to the variable sequence $x_1, \ldots, x_n$. For each variable $x_j$ its domain is given by
$\mathcal{L}_j =\{y\in \mathbb{Z} \mid    l_j \leq y \leq u_j \} \neq \emptyset$ and the domain of the entire variable vector is  $\mathcal{L} =\{\pmb{y} \in \mathbb{Z}^n \mid \forall j \in [n]: y_j \in \mathcal{L}_j\}$.  In the following, vectors are always written in bold font and the transpose sign for the scalar product between vectors is dropped for ease of notation. Let $\pmb{Q} \in \{\exists, \forall\}^n$ denote the vector of quantifiers.  We call each maximal consecutive subsequence in $\pmb{Q}$ consisting of identical quantifiers a \textit{block}. The quantifier corresponding to the $i$-th quantifier block is given by $Q^{(i)} \in \{\exists, \forall\}$. 
Let $\beta \in [n]$ denote the number of variable blocks.
With $\mathcal{L}^{(i)}$ we denote the corresponding domain of the $i$-th variable block as in $\mathcal{L}$.
At each move $\pmb{x}^{(i)}\in \mathcal{L}^{(i)}$, the corresponding player knows the settings of $\pmb{x}^{(1)}, \ldots, \pmb{x}^{(i-1)}$ before taking her decision. 
Each fixed vector $\pmb{x} \in \mathcal{L}$, that is, when the existential player has fixed the existential variables and the universal player has fixed the universal variables, is called \emph{a play}.  If $\pmb{x}$ satisfies the  \emph{existential constraint system}  $A^\exists \pmb{x} \leq \pmb{b}^\exists$, the existential player pays $\pmb{c} \pmb{x}$ to the universal player. If $\pmb{x}$ does not satisfy $A^\exists \pmb{x} \leq \pmb{b}^\exists$, we say \emph{the existential player loses} and the payoff is $+\infty$.
Therefore, it is the existential player's primary goal to ensure the fulfillment of the constraint system, while the universal player tries to violate some constraints. If the existential player is able to ensure that all constraints are fulfilled he tries to minimize $\pmb{c}\pmb{x}$, whereas the universal player tries to maximize her payoff. 

We consider QIPs with polyhedral uncertainty \cite{CG16,DissMichael} and therefore a \emph{universal constraint system} $A^\forall \pmb{x} \leq \pmb{b}^\forall$ is introduced, with $A^\forall_\exists=\pmb{0}$, \IE the submatrix of $A^\forall$ corresponding to existentially quantified variables is zero. Here the main goal of the universal player becomes satisfying this universal constraint system and therefore the universally quantified variables are restricted to a polytope. In particular,  a universal variable assignment must not make it impossible to satisfy the system $A^\forall \pmb{x} \leq \pmb{b}^\forall$. 
 Wit $A^\forall_\exists =\pmb{0}$ the system $A^\forall \pmb{x} \leq \pmb{b}^\forall$ restricts universal variables in such way that their range only depends on previous universal variables (cf. \cite{hartisch2019mastering}).

\begin{mydef}[QIP with Polyhedral Uncertainty]\label{Def::QIPPU}~\\
Let $\mathcal{L}$ and $\pmb{Q}$ be given with  $Q^{(1)}=Q^{(\beta)}=\exists$. Let $\pmb{c} \in \mathbb{Q}^n$ be the vector of objective coefficients, for which $\pmb{c}^{(i)}$ denotes the vector of coefficients belonging to block $i$.
The term $\pmb{Q} \circ \pmb{x} \in \D_All$ with the component-wise binding operator $\circ$ denotes the \emph{quantification sequence} $Q^{(1)}\pmb{x}^{(1)} \in {\D_All}^{(1)}\ Q^{(2)}\pmb{x}^{(2)}\in {\D_All}^{(2)}(\pmb{x}^{(1)})\ \ldots\ Q^{(\beta)} \pmb{x}^{(\beta)} \in {\D_All}^{(\beta)}(\pmb{x}^{(1)},\ldots,\pmb{x}^{(\beta-1)})$ such that every quantifier $Q^{(i)}$ binds the variables  $\pmb{x}^{(i)}$ of block $i$ ranging in their domain ${\D_All}^{(i)}(\pmb{x}^{(1)},\ldots,\pmb{x}^{(i-1)})$, with
${\D_All}^{(i)}(\tilde{\pmb{x}}^{(1)},\ldots,\tilde{\pmb{x}}^{(i-1)})= $ $$\begin{cases}
\mathcal{L}^{(i)} &\text{if } Q^{(i)} = \exists \\
\{\pmb{y} \in \mathcal{L}^{(i)} \mid \exists \pmb{x}=(\tilde{\pmb{x}}^{(1)},\ldots,\tilde{\pmb{x}}^{(i-1)},\pmb{y},{\pmb{x}}^{(i+1)},\ldots,{\pmb{x}}^{(\beta)}) \in \D_All \} 
&\text{if } Q^{(i)} = \forall  \, .
\end{cases}$$ 
\noindent
We call
\[
\resizebox{ \textwidth}{!} {$
\min\limits_{\pmb{x}^{(1)} \in  {\D_All}^{(1)}} \left(\pmb{c}^{(1)}\pmb{x}^{(1)}+ \max\limits_{\pmb{x}^{(2)} \in  {\D_All}^{(2)}} \left( \pmb{c}^{(2)}\pmb{x}^{(2)}  + \min\limits_{\pmb{x}^{(3)} \in  {\D_All}^{(3)}} \left(  \pmb{c}^{(3)}\pmb{x}^{(3)} + \ldots \min\limits_{\pmb{x}^{(\beta)} \in  {\D_All}^{(\beta)}}  \pmb{c}^{(\beta)}\pmb{x}^{(\beta)}\right) \right)\right)$}
\]
$$
 \text{s.t.}  \ \pmb{Q} \circ \pmb{x} \in \D_All  : A^\exists \pmb{x}\leq \pmb{b}^\exists
$$
a \emph{QIP with polyhedral uncertainty} given by the tuple $(A^\exists, A^\forall, \pmb{b}^\exists, \pmb{b}^\forall, \pmb{c}, \mathcal{L}, \pmb{Q})$.
\end{mydef}

We use $\mathcal{L}_\exists$ to describe the domain of the existentially quantified variables, given by their variables bounds as in $\mathcal{L}$. $\USet \neq \emptyset$ is the domain of universally quantified variables, \IE the uncertainty set, given by their domain and the universal constraint system. $\pmb{x_\exists}$ and $\pmb{x_\forall}$ denote the vectors only containing the existentially and universally quantified variables of game $\pmb{x}\in \D_All$, respectively. We call $\pmb{x_\forall} \in \USet$ a \textit{scenario} and refer to a partially filled universal variable vector as a \textit{subscenario}.
Additionally, we use $\mathcal{L}_{relax}$ to describe the domain given by $\mathcal{L}$ without the integrality condition. 
\subsection{Solving QIP via Game Tree Search }
A game tree can be used to represent the chronological order of all possible moves, given by the quantification sequence $\pmb{Q} \circ \pmb{x} \in \D_All$. The nodes in the game tree represent a partially assigned variable vector and  branches correspond to assignments of variables according to their variable domain. A solution of a QIP is a so-called winning (existential) strategy, that defines how to react to each legal move by the universal player, in order to ensure  \(A^\exists \pmb{x} \le \pmb{b}^\exists \). Hence, a solution is a subtree of the game tree with an exponential number of leaves with respect to the number of universal variables. If no such strategy exists the QIP is \textit{infeasible}.
If there is more than one solution, the objective function aims for a certain (the ``best'') one, whereat the value of a strategy is defined via the worst-case payoff at its leaves (see Stockman's Theorem \cite{Pijls}).
The play $\tilde{\pmb{x}}$ resulting in this leaf is called the principal variation \cite{Minimax}, which is the sequence of variable assignments being chosen during optimal play by both players.

The heart of the used search-based solver for 0/1-QIPs \cite{YasolACG17} is an arithmetic linear constraint database together with an alpha-beta algorithm, which has been successfully used in gaming programs, e.g. chess programs for many years \cite{KNUTH1975293,Hydra}.  The solver proceeds in two phases in order to find an optimal solution:
\begin{itemize}
\item feasibility phase: It is checked whether the instance has any solution. The solver acts like a quantified boolean formula (QBF) solver \cite{Cadoli2002,lonsing2010depqbf} with some extra abilities. Technically it performs a null window search \cite{Pearl80}.
\item optimization phase: The solution space is explored via alpha-beta algorithm in order to find the provable optimal solution. 
\end{itemize}
The alpha-beta algorithm  is enhanced by non-chronological backtracking and backward implication \cite{Qube,chen2001conflict}: 
when a contradiction is detected a reason in form of a clause is added to the constraint database and the search returns to the node where the found contradiction is no longer imminent. The solver deals with constraint learning on the so-called primal side as known from SAT- and QBF-solving (e.g. \cite{marques2009conflict,giunchiglia2002learning}), as well as with constraint learning on the dual side known from MIP (e.g. \cite{ceria1998cutting}). Several other techniques are implemented, e.g. restart strategies \cite{huang2007effect}, branching heuristics \cite{achterberg2005branching}, and pruning mechanisms \cite{hartisch2019novel}. 
Furthermore, relaxations are heavily used during the optimization phase: at every search node a relaxation is called in order to asses the quality of a branching decision, the satisfiability of the existential constraint system or for the generation of bounds.

\section{Enhanced Relaxations\label{Sec::Relax}}
\subsection{Relaxations for QIPs}

In case of a quantified program, besides relaxing the integrality of variables, the quantification sequence can be altered by changing the order or quantification of the variables. An LP-relaxation of a QIP can be built by dropping the integrality and also dropping universal quantification, \IE each variable is considered to be an existential variable with continuous domain. 
One major drawback of this LP-relaxation is that the worst-case perspective is lost by freeing the constraint system from having to be satisfied for any assignment of the  universally quantified variables: transferring the responsibility of universal variables to the existential player and solving the single-player game has nothing to do with the worst-case outcome in most cases. In order to strengthen this relaxation we use that for \textit{any} assignment of the universally quantified variables the constraint system must be fulfilled. Hence, fixing universally quantified variables according to some element of $\USet$ still yields a valid relaxation. This can be interpreted as knowing the  opponent moves beforehand and adapting one's own moves for this special play.

\begin{mydef}[LP-Relaxation with Fixed Scenario]\label{Def::LP_RELAX_FIXED}~\\
Let $P=(A^\exists, A^\forall, \pmb{b}^\exists, \pmb{b}^\forall, \pmb{c}, \mathcal{L}, \pmb{Q})$ and let $\pmb{\hat{x}}_\forall \in \USet$ be a fixed scenario. The LP $$\min \left\lbrace \pmb{c}\pmb{x} \mid \pmb{x}\in \mathcal{L}_{relax} \wedge \pmb{x_\forall} = \pmb{\hat{x}}_\forall \wedge A^\exists \pmb{x}\leq \pmb{b}^\exists \right\rbrace$$ is called the \emph{LP-relaxation with fixed scenario} $\pmb{\hat{x}}_\forall$ of $P$.
\end{mydef}

\begin{myprop}\label{PROP::EAS_RELAX}
Let $P=(A^\exists, A^\forall, \pmb{b}^\exists, \pmb{b}^\forall, \pmb{c}, \mathcal{L}, \pmb{Q})$ and let $R$ be the corresponding LP-relaxation with fixed scenario $\pmb{\hat{x}}_\forall  \in \USet$.  Then the following holds:
\begin{itemize}
\item[a)] If $R$ is infeasible, then also $P$ is infeasible.
\item[b)] If $R$ is feasible with optimal value $z_R$, then either $P$ is infeasible or $P$ is feasible with optimal value $z_P \geq z_R$, \IE $z_R$ constitutes a lower bound.
\end{itemize}
\end{myprop}
\begin{proof}~
\begin{itemize}
\item[a)]  If $R$ is infeasible then $$ \nexists \pmb{x}_\exists \in \mathcal{L}_\exists:\ A^\exists_\exists \pmb{x}_\exists \leq \pmb{b}^\exists-A^\exists_\forall \pmb{\hat{x}}_\forall\, ,$$
and since $\pmb{\hat{x}}_\forall \in \USet$ there cannot exist a winning strategy for $P$. As a gaming argument we can interpret this the following way: If there is some move sequence of the opponent we cannot react to in a victorious way---even if we know the sequence beforehand---the game is lost for sure.
\item[b)] Let $z_R=\pmb{c} \pmb{\hat{x}}$ be the optimal value of $R$, and let $\pmb{\hat{x}}_\exists$ be the corresponding fixation of the existential variables. It is \begin{equation}\label{Eq::ArgMinE}
\pmb{\hat{x}}_\exists = \argmin_{\pmb{x_\exists} \in \mathcal{L}_\exists}\left\lbrace \pmb{c_\exists} \pmb{x_\exists} \mid  A^\exists_\exists \pmb{x_\exists} \leq \pmb{b}^\exists - A^\exists_\forall \pmb{\hat{x}}_\forall\right\rbrace\, .
\end{equation} If $P$ is feasible, scenario  $\pmb{\hat{x}}_\forall$ must also be present in the corresponding winning strategy. Let $\pmb{\tilde{x}}$ be the corresponding play, \IE $\pmb{\tilde{x}}_\forall = \pmb{\hat{x}}_\forall$. With Equation \eqref{Eq::ArgMinE} obviously $z_R=\pmb{c} \pmb{\hat{x}}\leq \pmb{c}\pmb{\tilde{x}}$ and thus with Stockman's Theorem \cite{Pijls} $z_R \leq z_P$. 
\end{itemize}
\end{proof}

As we will show in Section \ref{Sec::Exp} adding a scenario to the LP-relaxation already massively speeds up the search process compared to the use of the standard LP-relaxation. However, partially incorporating the multistage nature into a relaxation should yield even better bounds. Therefore, we reintroduce the original order of the variables while only taking a subset of scenarios $S \subseteq \USet$ into account.

\begin{mydef}[$S$-Relaxation]\label{Def::SRelax}~\\
Given $P=(A^\exists, A^\forall, \pmb{b}^\exists, \pmb{b}^\forall, \pmb{c}, \mathcal{L}, \pmb{Q})$. Let $S \subseteq \USet$ and let $\mathcal{L}_S=\{\pmb{x} \in \mathcal{L}_{relax} \mid \pmb{x_\forall} \in S\}$.
We call

\[
\resizebox{ \textwidth}{!} {$
\min\limits_{\pmb{x}^{(1)} \in \mathcal{L}_S^{(1)}}\left( \pmb{c}^{(1)}\pmb{x}^{(1)}+ \max\limits_{\pmb{x}^{(2)} \in \mathcal{L}_S^{(2)}} \left( \pmb{c}^{(2)}\pmb{x}^{(2)} + \min\limits_{\pmb{x}^{(3)} \in \mathcal{L}_S^{(3)}} \left( \pmb{c}^{(3)}\pmb{x}^{(3)} + \ldots \min\limits_{\pmb{x}^{(\beta)} \in \mathcal{L}_S^{(\beta)}} \pmb{c}^{(\beta)} \pmb{x}^{(\beta)}\right)\right)\right) $}
\]
\begin{equation}
\textnormal{s.t.}\ Q \circ \pmb{x} \in \mathcal{L}_S:\ A^\exists \pmb{x} \leq \pmb{b}^\exists
\end{equation}	 
the \emph{$S$-relaxation} of $P$.
\end{mydef}

\begin{myprop}
Let  $P=(A^\exists, A^\forall, \pmb{b}^\exists, \pmb{b}^\forall, \pmb{c}, \mathcal{L}, \pmb{Q})$ be feasible and let $R$ be the $S$-relaxation with $\emptyset \neq S \subseteq \USet$ and optimal value $\tilde{z}_R$. Then $\tilde{z}_R$ is a lower bound on the optimal value $\tilde{z}_P$ of $P$, i.e. $\tilde{z}_R \leq \tilde{z}_P$.
\end{myprop}
\begin{proof}
Again we use a gaming argument: with $S \subseteq \USet$  the universal player is restricted to a subset of her moves in problem $R$, while the existential player is no longer restricted to use integer values. Furthermore, any strategy for $P$ can be mapped to a strategy for the restricted game $R$. Hence, the optimal strategy for $R$ is either part of a strategy for $P$ or it is an even better strategy, as the existential player does not have to cope with the entire variety of the universal player's moves. Therefore,  $\tilde{z}_R \leq \tilde{z}_P$.
\end{proof}

In general, $\USet$ has exponential size with respect to the number of universally quantified variables. Therefore, the main idea is to keep $S$ a rather small subset of $\USet$. This way the DEP of the $S$-relaxation---which is a standard LP--- remains easy to handle for standard LP solvers.
 
 \begin{myex}\label{Example::S_Relax}
Consider the following binary QIP (The min/max alternation in the objective is omitted for clarity):
$$\begin{array}{rrrrrl}
\min & -2x_1&+x_2 &-x_3& -x_4 \\
\textnormal{s.t.} &\exists\, x_1 & \forall\, x_2 &\exists\, x_3& \forall \, x_4 &\in \{0,1\}^4:\\
&x_1 & +x_2 &+x_3& +x_4 & \leq 3\\
& & -x_2 &-x_3& +x_4 & \leq 0
\end{array} $$
The optimal first-stage solution is $\tilde{x}_1=1$, the principal variation is $(1,1,0,0)$ and hence the optimal value is $-1$. Let $S=\{(1,0),(1,1)\}$ be a set of scenarios. The two LP-relaxations with fixed scenario accoring to the two scenarios in $S$ are shown in Table \ref{Tab::ExampleRelaxS}.
\begin{table}[h!]
\centering
\footnotesize
\caption{Solutions of the single LP-relaxations with fixed scenarios.}\label{Tab::ExampleRelaxS}
\begin{tabular}{p{2cm}p{4.5cm}p{4.5cm}}
\toprule
scenario &  $x_2=1$, $x_4=0$ &  $x_2=1$, $x_4=1$\\\midrule
relaxation&$\begin{array}{rrrrrl}
\min & -2x_1 &-x_3& +1\\
\textnormal{s.t.} &x_1 &+x_3 & \leq 2\\
&  &-x_3&   \leq 1
\end{array} $
&
$\begin{array}{rrrrrl}
\min & -2x_1 &-x_3& + 0\\
\textnormal{s.t.} &x_1 &+x_3 & \leq 1\\
&  &-x_3&   \leq 0
\end{array} $\\\addlinespace[.1cm]
solution & $x_1=1$, $x_3=1$ & $x_1=1$, $x_3=0$ \\\addlinespace[.1cm]
objective & -2&-2\\\bottomrule
\end{tabular}
\end{table}
Both yield the optimal first stage solution of setting $x_1$ to one. Now consider the DEP of the $S$-relaxation in which $x_{3(\tilde{x}_2)}$ represents the assignment of $x_3$ after $x_2$ is set to $\tilde{x}_2$:
$$\begin{matrix*}[l]
\min\  k \\
\left.\begin{matrix*}[r]\textnormal{s.t.}&-2x_1 &-x_{3(1)}& +1&\multicolumn{1}{l}{\leq k}\\
&x_1 &+x_{3(1)}& & \multicolumn{1}{l}{\leq 2}\\
&  &-x_{3(1)}& & \multicolumn{1}{l}{\leq 1}\\ \end{matrix*}\quad \right\rbrace \text{Scenario $(1,0)$}\\
\left.\begin{matrix*}[r]\color{white}\text{s.t.}&-2x_1 &-x_{3(1)}& +0 &\multicolumn{1}{l}{\leq k}\\
&x_1 &+x_{3(1)}& & \multicolumn{1}{l}{\leq 1}\\
&  &-x_{3(1)}& & \multicolumn{1}{l}{\leq 0}\end{matrix*}\quad \right\rbrace \text{Scenario $(1,1)$}\\
\left.\begin{matrix*}[r]\color{white}\text{s.t.}\color{black} &\color{white}-2\color{black} x_1&,\hspace{3.65pt} x_{3(1)}&\in [0,1]\end{matrix*}\right.
\end{matrix*} $$
In the $S$-relaxation it is ensured that variables following equal sub-scenarios are set to the same value. As $x_2$ is set to $1$ in each considered scenario in $S$, $x_3$ must be set to the same value in both cases. The solution of the DEP is $x_1=1$, $x_{3(1)}=0$ and $k=-1$. Thus, the $S$-relaxation yields the lower bound -1 for the original QIP. This is not only a better bound than the one obtained by the two LP-relaxations with individually fixed scenarios but it is also a tight bound.
\end{myex}

\subsection{Scenario Selection\label{Sec::ScenSel}}
Both for the LP-relaxation with fixed scenario as well as the $S$-relaxation the selection of scenarios is crucial. For the $S$-relaxation additionally the size of the scenario set $S$ affects its performance, in particular if too many scenarios are chosen, solving the relaxation might consume too much time. 
We use three heuristics to collect information on universal variables during the search:

\paragraph{VSIDS heuristic \cite{moskewicz2001chaff}.} Each variable in each polarity has a counter, initialized to 0. When a clause is added, due to a found conflict, the counter associated with each literal is incremented. Periodically, all counters are divided by a constant. 
\paragraph{Killer heuristic \cite{Killer}.}
When a conflict is found during the search the current assignment of universal variables---and thus the (sub)scenario leading to this conflict---is stored in the $\mathtt{killer}$ vector. This is a short-term information and is overwritten as soon as a new conflict is found.
\paragraph{Scenario frequency.}
For each scenario and subscenario the frequency of their occurrence during the search is stored.\\

The LP-relaxation with fixed scenario is implemented as follows: before calling the LP solver in a decision node, all variable bounds must be updated according to the current node anyway. When doing so (yet unassigned) universally quantified variables are set as in Algorithm \ref{Algo}.
Hence, the considered scenario is adapted in every decision node based on the latest heuristic information.

\begin{algorithm}[h!]
  \caption{Building a scenario}\label{Algo}
  \begin{algorithmic}[1]
  \For{each universal variable block $i \in \{1,\ldots,\beta \mid Q^{(\beta)}=\forall\}$  }{
\For{each unassigned variable $x_j$ in block $i$, in random order }{
\If{ $\mathtt{killer[j]}\neq \mathtt{undefined}$\label{Line:3}} 
    	 		 $\mathtt{Value} = \mathtt{killer[j]}$
    	 	\Else{\label{Line:4}}
    	 	$\mathtt{Value} = \argmax_{p\in\{0,1\}}{\mathtt{VSIDS}[j][p]}$
    	 \EndIf
    	 \If{ setting $x_j$ to $\mathtt{Value}$ is legal according to $\D_All^{(i)}$ } 
    	 		$x_j = \mathtt{Value}$
    	 	\Else {}
    	 	  $x_j = 1-\mathtt{Value}$
    	 \EndIf
  }
  \EndFor
  }
\EndFor
\end{algorithmic}
\end{algorithm}

The $S$-relaxation is adapted at each restart. The scenario set $S$ is rebuilt by considering the $\bar{S} \in \mathbb{N}$ most frequently used (sub)scenarios. Subscenarios are extended to a full scenario according to Algorithm \ref{Algo}. Even though starting with $\bar{S}$ (sub)scenarios, $S$ often contains fewer unique scenarios, as extending a subscenario may result in a scenario already contained in $S$.

Furthermore, our implementation merges the LP-relaxation with fixed scenario into the $S$-relaxation: the final relaxation takes all scenarios in the scenario set $S$, as well as one additional scenario that can be updated at each decision node into account. Hence, the used relaxation in fact reflects $|S|+1$ scenarios and in case of $S=\emptyset$ the LP-relaxation with fixed scenario remains. The DEP of this final relaxation is built and solved with an external LP solver.

The $S$-relaxation is currently only used while the search is in the very first variable block, \IE as soon as all variables of the first block are assigned, only the LP-relaxation with fixed scenario is used. The reason why this relaxation is no longer used in later variable blocks is that then universally quantified variables are already fixed according to the current search node. Hence, some scenarios in $S$ are no longer relevant as they refer to other parts of the search tree. Therefore, in order to use the $S$-relaxation in higher blocks it needs to be rebuilt each time a universal variable block is bypassed.

\section{Experiments\label{Sec::Exp}}
\subsection{Problem Descriptions}
We conduct experiments on three different QIPs with polyhedral uncertainty.  For a more detailed discussion on the problem formulations we refer to \cite{DissMichael}.

\paragraph{Multistage robust selection.} The goal is to select $p$ out of $n$ items with minimal costs. In an initial (existential) decision stage a set of items can be selected for fixed costs. Then, in a universal decision stage, one of $N \in \mathbb{N}$ cost scenario is disclosed. In the subsequent existential decision stage further items can be selected for the revealed costs. The latter two stages are repeated iteratively $T \in \mathbb{N}$ times. Hence, there are  $2T+1$ variable blocks.
\paragraph{Multistage robust assignment.} The goal is to find a perfect matching for a bipartite graph  $G = (V, E)$,  $V = A \cup B$, $n=|A| = |B|$, with minimal costs. In an initial (existential) decision stage a set of edges can be selected for fixed costs. Then, in a universal decision stage, one of $N \in \mathbb{N}$ cost scenario is disclosed. In the subsequent existential decision stage further edges can be selected for the revealed costs. Those two stages are repeated iteratively $T \in \mathbb{N}$ times. Both for the selection and the assignment problem, a universal constraint system is used to force the universally quantified variables to reveal exactly one scenario per period.
\paragraph{Multistage robust runway scheduling.} Each airplane $i \in A$ has to be assigned to exactly one time slot $j \in S$ and at most $b\in \mathbb{N}$ airplanes can be assigned to one time slot (as there are only $b$ runways). As soon as the (uncertain) time window in which the airplane can land is disclosed by universally quantified variables, the initial plan has to be adapted. The goal is to find an initial schedule that can be adapted according to the later disclosed time windows with optimal worst-case costs, as for each slot that the airplane is moved away from its originally planned time slot, a cost is incurred. The time window disclosure occurs in $T \in \mathbb{N}$ periods: the airplanes are partitioned into $T$ groups and after the initial schedule is fixed the time windows are disclosed for one group after the other. After each disclosure the schedule for the current group of airplanes has to be fixed right away, before knowing the time windows for the subsequent groups. The universal constraint system contains a single constraint, demanding that the disclosed time windows are comprised of $3$ time slots on average. 

\subsection{Computational Results}

The used solver 
 utilizes CPLEX (12.6.1) as its black-box LP solver to solve the relaxations and all experiments were run with AMD Ryzen 9 5900X processors.

First we provide details on the benefit of utilizing the LP-relaxation with fixed scenario as given in Definition \ref{Def::LP_RELAX_FIXED} compared to the standard LP at each decision node. Therefore, we consider the following testset:
\begin{itemize}
\item $350$ selection instances with $n=10$ items, $N=4$ scenarios per period and $T \in \{1,\ldots,7\}$ periods
\item $1350$ assignment instances with $n\in \{4,5,6\}$, $N \in \{2^1,2^2,2^3\}$ scenarios per period and $T \in \{1,\ldots,3\}$ periods
\item $270$ runway scheduling instances with $A \in \{4,5,6\}$ planes, $b=3$ runways, $S\in \{5,\ldots,10\}$ time slot and $T \in \{1,\ldots,3\}$ periods
\end{itemize}
In Table \ref{Tab::RuntimeOnlyFixLP}, as one of our major results, the overall runtimes when using the basic LP-relaxation and the LP-relaxation with fixed scenario are display. 
\begin{table}
\centering
\caption{Overall runtime (in seconds) when only using the standard LP-relaxation vs. the LP-relaxation with fixed scenario. \label{Tab::RuntimeOnlyFixLP}}
\begin{tabular}{lrrrrrr}
\toprule
used relaxation &~\hspace{.1cm}~& selection &~\hspace{.1cm}~& assignment &~\hspace{.1cm}~& runway\\\midrule
LP&&		\numprint{29501}	&&	\numprint{7152}	&&	\numprint{12902}\\
LP with fixed scenario &&	\numprint{348}	&&	\numprint{837}	&&	\numprint{4520}	\\\bottomrule
\end{tabular}
\end{table}
In each case, explicitly setting the universally quantified to a fixed scenario results in a massive speedup that is most impressive for the selection instances. This emphasizes, that partially incorporating the worst-case nature of the underlying problem into the basic LP-relaxation is clearly beneficial and does not have any negative side effects: the bounds of the variables in the LP-relaxation have to be updated at each search node  anyway and fixing the universally quantified variables even decreases the number of free variables in the resulting LP. 

We now investigate how the use of the more sophisticated $S$-relaxation in the first variable block changes the solver's behavior. Therefore, the scenario set $S$ is built from $\bar{S}=2^i$ (sub)scenarios, with $i \in \{0,\ldots,6\}$. In case of $\bar{S}=0$ only the LP-relaxation with fixed scenario is utilized in the first variable block. The used testset consists of the following instances

\begin{itemize}
\item $1050$ selection instances with $n\in \{10,20,30\}$ items, $N=4$ scenarios per period and $T \in \{1,\ldots,7\}$ periods
\item $450$ assignment instances with $n\in \{7\}$, $N \in \{2^1,2^2,2^3\}$ scenarios per period and $T \in \{1,\ldots,3\}$ periods
\item $360$ runway scheduling instances with $A \in \{4,5,6,7\}$ planes, $b=3$ runways, $S\in \{5,\ldots,10\}$ time slot and $T \in \{1,\ldots,3\}$ periods
\end{itemize}
As one indicator we consider the number of decision nodes visited during the optimization phase of the search. We denote $N(i,\bar{S})$  the number of visited decision nodes when solving instance $i$ with $\bar{S}$ scenarios used to build the corresponding $S$-relaxation. We compare each run with $\bar{S}=2^i$ to the basic run with $\bar{S}=0$ by considering the relative difference $D_r(i)=\frac{N(i,\bar{S})-N(i,0)}{\max(N(i,\bar{S}),N(i,0))}$. If $N(i,\bar{S})-N(i,0)<0$, i.e. if fewer decision nodes were visited while using the $S$-relaxation, $D_r(i)$ is negative, with its absolute value indicating the percentage savings. Similarly, if $N(i,\bar{S})-N(i,0)>0$, $D_r(i)$ is positive. The data on all instances is cumulated in Figure \ref{fig:BoxPlots} showing the corresponding box plots\footnote{Box plots are created using the macro psboxplot of the \LaTeX~package pst-plot15. The interquantile
range factor, defining the area of outliers, is set to 1.5 by default.}. 
\begin{figure}
     \centering
     \begin{subfigure}[b]{0.3\textwidth}
         \centering
         \includegraphics[width=\textwidth]{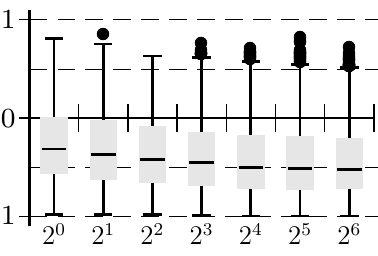}
         \caption{Selection}
         \label{fig:y equals x}
     \end{subfigure}
     \hfill
     \begin{subfigure}[b]{0.3\textwidth}
         \centering
         \includegraphics[width=\textwidth]{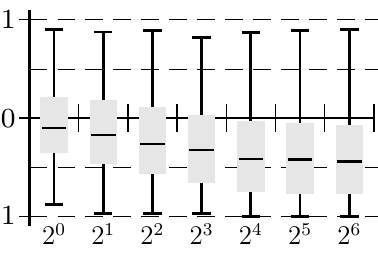}
         \caption{Assignment}
         \label{fig:three sin x}
     \end{subfigure}
     \hfill
     \begin{subfigure}[b]{0.3\textwidth}
         \centering
          \includegraphics[width=\textwidth]{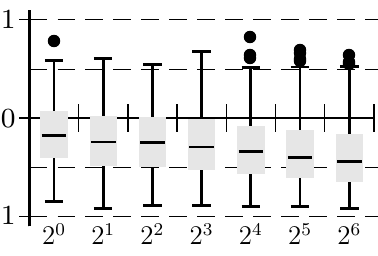}
         \caption{Runway}
         \label{fig:five over x}
     \end{subfigure}
        \caption{Boxplots of the $D_r$ values for all three testset and $\bar{S}\in \{2^0,\ldots, 2^6\}$}
        \label{fig:BoxPlots}
\end{figure}
It can be seen that the median of the relative difference values is always negative and  tends to decrease when more scenarios are considered in $S$, i.e. the larger the scenario set the fewer decision nodes have to be visited during the search. Note that compared to the box plots for the selection and runway  instances, for the assignment instances the upper whisker remains at a rather high level. But does a decreasing number of visited decision nodes also result in a lower runtime? For now consider the columns of Table~\ref{Tab::RuntimeSRelax} representing the heuristic scenario selection as presented in Section \ref{Sec::ScenSel}.
\begin{table}
\centering
\caption{Overall runtime (in seconds) when using the $S$-relaxation with heuristic and random scenario selection \label{Tab::RuntimeSRelax}}
\begin{tabular}{p{.5cm}R{1.44cm}R{1.44cm}p{.4cm}R{1.44cm}R{1.44cm}p{.4cm}R{1.44cm}R{1.44cm}}
\toprule
&\multicolumn{2}{r}{selection instances}&&\multicolumn{2}{r}{assignment instances}&
&\multicolumn{2}{r}{runway instances}\\
$\bar{S}$&heuristic & random &&heuristic & random &&heuristic & random\\\midrule 
$0$		& \numprint{12561}	 & \numprint{53348}	&& \numprint{2091}	 	& \numprint{1853}		&& \numprint{33335}		&	\numprint{32401}	\\ 
$2^0$	& \numprint{11324}	 	& \numprint{35316}	&& \numprint{2111}		& \numprint{1865}		&& \numprint{29313}		& \numprint{30418}\\ 
$2^1$	& \numprint{9900} 		& \numprint{30970}	&& \numprint{2022}		& \numprint{2046}		&& \numprint{25876}		&	\numprint{26412}	\\ 
$2^2$	&	\numprint{9700} 	& \numprint{31158}	&& \numprint{2210}		& \numprint{2232}		&& \numprint{25876}		&	\numprint{26101}	\\ 
$2^3$	&	\numprint{9394} 	& \numprint{29087}	&& \numprint{2220}		& \numprint{2708}		&& \numprint{23915}	 	&	\numprint{24795}	\\ 
$2^4$	& 	\numprint{9030}		& \numprint{27503}	&& \numprint{2931}		& \numprint{3718}		&& \numprint{23958}		&	\numprint{24860}	\\ 
$2^5$	& \numprint{8843}		& \numprint{26857}	&& \numprint{4223}		& \numprint{7300}		&& \numprint{21788}		&	\numprint{26777}	\\ 
$2^6$	&	\numprint{9149}	& \numprint{26590}	&& \numprint{8632}		& \numprint{17400}		&& \numprint{23073} 	&\numprint{30292}	\\ 
\bottomrule
\end{tabular}
\end{table}
Both for the selection and the runway scheduling problem the overall runtimes tend to decrease when $\bar{S}$ increases. Compared to only using the LP-relaxation with fixed scenario ($\bar{S}=0$), the runtimes decreased up to about  30\% and 35\% for the selection and runway scheduling instances, respectively. The slightly increased runtime for $\bar{S}=64$ indicates that the solution time of such a large relaxation can no longer be compensated by fewer visited decision nodes.  For the assignment instances, however, the overall runtime increases, up to a factor of four times the runtime when solely using the LP-relaxation with fixed scenario. Hence, even though fewer nodes are visited, the time it takes to process and generate information at these nodes increases considerably for this type of problem. 

In Table  \ref{Tab::RuntimeSRelax} we additionally provide information on how  well our scenario building routine performs on the considered testset. Therefore, instead of extending the $\bar{S}$ most frequently visited (sub)scenarios via Algorithm \ref{Algo}, the scenario set $S$ now contains $\bar{S}$ random scenarios.  Similiarly, for the LP-relaxation with fixed scenario, we replace the heuristic $\mathtt{Value}$ selection in lines \ref{Line:3} and \ref{Line:4} of Algorithm \ref{Algo} by randomly assigning the value $0$ or $1$. Note, however, that even though the killer and VSIDS information is neglected while building the relaxation, it is still utilized in other situations during the search.  The overall runtimes are shown in the according columns of Table  \ref{Tab::RuntimeSRelax}. For the selection problem, randomly selecting the scenario results in a runtime about three times longer compared to using the heuristic selection process. For the runway instances, our heuristic also slightly outperforms the use of random scenarios. For the assignment instances the random scenario selection tends to be more favorable when only few scenarios are involved.

\section{Conclusion and Outlook\label{Sec::Concl}}
We investigated how adaptive relaxations influence our search-based solution algorithm for multistage robust optimization problems.
Our experimental results show that incorporating a single scenario in the standard LP-relaxation significantly speeds up the search process and clearly dominates the basic LP-relaxation. Furthermore, the use of the $S$-relaxation which incorporates a subset of scenarios in a slim DEP, considerably decreases the number of visited nodes, even if only utilized in the very first variable block. While this smaller search space also resulted in a faster solution time for multistage selection and runway scheduling problems, the solution time tended to increase for multistage assignment instances. Additionally, we showed that our scenario selection heuristic outperforms a random scenario selection. 

Several research questions arise from the presented experimental results. Is it possible to improve the heuristic scenario selection? Currently our heuristic focuses on including seemingly harmful scenarios but does not consider the diversity of the scenario set $S$, which might be one reason why using random scenarios already works quite well on specific problems. In contrast to our currently implemented search-information-based scenario selection heuristic, we find it interesting to deploy AI methods in order to classify scenarios as relevant and irrelevant for general QIP. Additionally, relevant characteristics of instances have to be found, in order to dynamically adjust  the size of the used scenario set $S$. Furthermore, deploying the $S$-relaxation in all---not only the very first---variable blocks is a very promising yet challenging task, as the implementation of such a frequently modified $S$-relaxation must be done carefully. In this case having the ability to update all considered scenarios in each decision node is also of interest, in particular as our results showed that having few scenarios in the relaxation is already very beneficial.

%
%
 \bibliographystyle{splncs04}
 \bibliography{mybibliography}

\begin{thebibliography}{10}
\providecommand{\url}[1]{\texttt{#1}}
\providecommand{\urlprefix}{URL }
\providecommand{\doi}[1]{https://doi.org/#1}

\bibitem{achterberg2005branching}
Achterberg, T., Koch, T., Martin, A.: Branching rules revisited. Operations
  Research Letters  \textbf{33}(1),  42--54 (2005)

\bibitem{Killer}
Akl, S., Newborn, M.: The principal continuation and the killer heuristic. In:
  Proceedings of the 1977 annual conference, {ACM} '77, Seattle, Washington,
  USA. pp. 466--473 (1977)

\bibitem{Balas2001}
Balas, E.: Projection and lifting in combinatorial optimization. In:
  J{\"u}nger, M., Naddef, D. (eds.) Computational Combinatorial Optimization:
  Optimal or Provably Near-Optimal Solutions. pp. 26--56. Springer (2001)

\bibitem{Ben-Tal}
Ben-Tal, A., Ghaoui, L.E., Nemirovski, A.: Robust Optimization. Princeton
  University Press (2009)

\bibitem{bertsimas2016multistage}
Bertsimas, D., Dunning, I.: Multistage robust mixed-integer optimization with
  adaptive partitions. Operations Research  \textbf{64}(4),  980--998 (2016)

\bibitem{bertsimas2015design}
Bertsimas, D., Georghiou, A.: Design of near optimal decision rules in
  multistage adaptive mixed-integer optimization. Operations Research
  \textbf{63}(3),  610--627 (2015)

\bibitem{Cadoli2002}
Cadoli, M., Schaerf, M., Giovanardi, A., Giovanardi, M.: An algorithm to
  evaluate quantified boolean formulae and its experimental evaluation. Journal
  of Automated Reasoning  \textbf{28}(2),  101--142 (2002)

\bibitem{Minimax}
Campbell, M., Marsland, T.: A comparison of minimax tree search algorithms.
  Artificial Intelligence  \textbf{20}(4),  347--367 (1983)

\bibitem{ceria1998cutting}
Ceria, S., Cordier, C., Marchand, H., Wolsey, L.A.: Cutting planes for integer
  programs with general integer variables. Math. programming  \textbf{81}(2),
  201--214 (1998)

\bibitem{chen2001conflict}
Chen, X., Van~Beek, P.: Conflict-directed backjumping revisited. Journal of
  Artificial Intelligence Research  \textbf{14},  53--81 (2001)

\bibitem{delage2015robust}
Delage, E., Iancu, D.A.: Robust multistage decision making. In: The operations
  research revolution, pp. 20--46. INFORMS (2015)

\bibitem{Hydra}
Donninger, C., Lorenz, U.: The chess monster {H}ydra. In: Field Programmable
  Logic and Application. pp. 927--932. Springer (2004)

\bibitem{YasolACG17}
Ederer, T., Hartisch, M., Lorenz, U., Opfer, T., Wolf, J.: Yasol: An open
  source solver for quantified mixed integer programs. In: 15th International
  Conference on Advances in Computer Games, {ACG} 2017. pp. 224--233. Springer
  (2017)

\bibitem{ederer2011quantified}
Ederer, T., Lorenz, U., Martin, A., Wolf, J.: Quantified linear programs: a
  computational study. In: European Symposium on Algorithms. pp. 203--214.
  Springer (2011)

\bibitem{Qube}
Giunchiglia, E., Narizzano, M., Tacchella, A.: Backjumping for quantified
  boolean logic satisfiability. Artificial Intelligence  \textbf{145}(1),
  99--120 (2003)

\bibitem{giunchiglia2002learning}
Giunchiglia, E., Narizzano, M., Tacchella, A., et~al.: Learning for quantified
  {Boolean} logic satisfiability. In: Proceedings of the AAAI Conference on
  Artificial Intelligence. pp. 649--654 (2002)

\bibitem{goerigk2021multistage}
Goerigk, M., Hartisch, M.: Multistage robust discrete optimization via
  quantified integer programming. Under revision at Computers and Operations
  Research, arXiv preprint arXiv:2009.12256  (2021)

\bibitem{GuptaA}
Gupta, A., P\'{a}l, M., Ravi, R., Sinha, A.: Boosted sampling: Approximation
  algorithms for stochastic optimization. In: Proceedings of the Thirty-sixth
  Annual ACM Symposium on Theory of Computing. pp. 417--426. STOC '04 (2004)

\bibitem{DissMichael}
Hartisch, M.: Quantified Integer Programming with Polyhedral and
  Decision-Dependent Uncertainty. Ph.D. thesis, University of Siegen, Germany
  (2020). \doi{10.25819/ubsi/4841}

\bibitem{CG16}
Hartisch, M., Ederer, T., Lorenz, U., Wolf, J.: Quantified integer programs
  with polyhedral uncertainty set. In: Computers and Games - 9th International
  Conference, {CG} 2016. pp. 156--166. Springer (2016)

\bibitem{hartisch2019mastering}
Hartisch, M., Lorenz, U.: Mastering uncertainty: Towards robust multistage
  optimization with decision dependent uncertainty. In: Pacific Rim
  International Conference on Artificial Intelligence. pp. 446--458. Springer
  (2019)

\bibitem{hartisch2019novel}
Hartisch, M., Lorenz, U.: A novel application for game tree search-exploiting
  pruning mechanisms for quantified integer programs. In: Advances in Computer
  Games. pp. 66--78. Springer (2019)

\bibitem{hemmi2018recursive}
Hemmi, D., Tack, G., Wallace, M.: A recursive scenario decomposition algorithm
  for combinatorial multistage stochastic optimisation problems. In:
  Proceedings of the AAAI Conference on Artificial Intelligence. vol.~32 (2018)

\bibitem{huang2007effect}
Huang, J., et~al.: The effect of restarts on the efficiency of clause learning.
  In: IJCAI. vol.~7, pp. 2318--2323 (2007)

\bibitem{KM05}
Kall, P., Mayer, J.: Stochastic linear programming. Models, theory and
  computation. Springer (2005)

\bibitem{KNUTH1975293}
Knuth, D., Moore, R.: An analysis of alpha-beta pruning. Artificial
  Intelligence  \textbf{6}(4),  293--326 (1975)

\bibitem{lonsing2010depqbf}
Lonsing, F., Biere, A.: Depqbf: A dependency-aware qbf solver. Journal on
  Satisfiability, Boolean Modeling and Computation  \textbf{7}(2-3),  71--76
  (2010)

\bibitem{lorenz2015solving}
Lorenz, U., Wolf, J.: Solving multistage quantified linear optimization
  problems with the alpha--beta nested benders decomposition. EURO Journal on
  Computational Optimization  \textbf{3}(4),  349--370 (2015)

\bibitem{marques2009conflict}
Marques-Silva, J., Lynce, I., Malik, S.: Conflict-driven clause learning {SAT}
  solvers. In: Handbook of Satisfiability, pp. 131--153. ios Press (2009)

\bibitem{mercier2007performance}
Mercier, L., Van~Hentenryck, P.: Performance analysis of online anticipatory
  algorithms for large multistage stochastic integer programs. In: IJCAI. pp.
  1979--1984 (2007)

\bibitem{moskewicz2001chaff}
Moskewicz, M.W., Madigan, C.F., Zhao, Y., Zhang, L., Malik, S.: Chaff:
  Engineering an efficient sat solver. In: Proceedings of the 38th annual
  Design Automation Conference. pp. 530--535 (2001)

\bibitem{Papadimitriou}
Papadimitriou, C.: Games against nature. Journal of Computer and System
  Sciences  \textbf{31}(2),  288--301 (1985)

\bibitem{Pearl80}
Pearl, J.: Scout: A simple game-searching algorithm with proven optimal
  properties. In: Proceedings of the First AAAI Conference on Artificial
  Intelligence. pp. 143--145. AAAI'80, AAAI Press (1980)

\bibitem{Pijls}
Pijls, W., de~Bruin, A.: Game tree algorithms and solution trees. Theoretical
  Computer Science  \textbf{252}(1),  197--215 (2001)

\bibitem{schaeffer1989history}
Schaeffer, J.: The history heuristic and alpha-beta search enhancements in
  practice. IEEE transactions on pattern analysis and machine intelligence
  \textbf{11}(11),  1203--1212 (1989)

\bibitem{shapiro2011dynamic}
Shapiro, A.: A dynamic programming approach to adjustable robust optimization.
  Operations Research Letters  \textbf{39}(2),  83--87 (2011)

\bibitem{Subramani_Integer}
Subramani, K.: Analyzing selected quantified integer programs. In: Automated
  Reasoning: Second International Joint Conference, IJCAR 2004. vol.~3097, pp.
  342--356. Springer (2004)

\bibitem{takriti2004robust}
Takriti, S., Ahmed, S.: On robust optimization of two-stage systems. Math.
  Programming  \textbf{99}(1),  109--126 (2004)

\end{thebibliography}

\end{document}